\documentclass{amsart}
\usepackage{amssymb}

\title[Variation of the canonical height in several variables]{Variation of the canonical height for polynomials in several variables}
\author{Patrick Ingram}
\address{Colorado State University, Fort Collins, Colorado, USA}
\email{pingram@math.colostate.edu}
\date{\today}

\newcommand{\QQ}{\mathbb{Q}}
\newcommand{\ZZ}{\mathbb{Z}}
\newcommand{\CC}{\mathbb{C}}
\newcommand{\RR}{\mathbb{R}}

\newcommand{\PP}{\mathbb{P}}
\renewcommand{\AA}{\mathbb{A}}

\newcommand{\Ocal}{\mathcal{O}}

\newcommand{\Supp}{\operatorname{Supp}}

\newcommand{\ord}{\operatorname{ord}}

\newcommand{\Div}{\operatorname{Div}}

\newcommand{\basin}{\mathcal{B}}

\renewcommand{\phi}{\varphi}
\renewcommand{\epsilon}{\varepsilon}
\newcommand{\var}{C}

\newtheorem{theorem}{Theorem}
\newtheorem{lemma}[theorem]{Lemma}

\newtheorem{corollary}[theorem]{Corollary}

\theoremstyle{remark}
\newtheorem{remark}{Remark}

\theoremstyle{definition}

\begin{document}

\begin{abstract}
Let $K$ be a number field, let $X/K$ be a curve, let $f_t:\PP^N\to\PP^N$ be a family of morphisms over $X$, and let $P:X\to \PP^N$ be a morphism. Writing $\hat{h}_{f_t}$ for the canonical height associated to $f_t$, it follows from a result of Call and Silverman that $\hat{h}_{f_t}(P_t)/h_X(t)=\hat{h}_{f_\eta}(P_\eta)+o(1)$, where $\eta$ is the generic point of the curve. Assuming that there is a hyperplane $H$ such that $f_\eta^*H=dH$, and such that the restriction of the family to $H$ is isotrivial, we improve this estimate to
\[\hat{h}_{f_t}(P_t)=\hat{h}_{f_\eta}(P_\eta)h_X(t)+O(1),\]
for a particular Weil height $h$ on $X$. In the one-variable case, this reduces to an earlier result of the author.
\end{abstract}
\maketitle

Let $f:\PP^N\to\PP^N$ be a morphism of degree $d\geq 2$, defined over a number field $K$.  There is, associated to $f$, a canonical height function $\hat{h}_f:\PP^N(\overline{K})\to\RR$ which is roughly analogous to the N\'{e}ron-Tate height on an abelian variety, and reveals much about the arithmetic dynamics of $f$. It is natural to ask how this function varies as one varies the morphism to which it is associated.

Call and Silverman~\cite{call-silv} have shown that if $t\mapsto f_t$ is an algebraic family of endomorphisms of $\PP^N$, over a base curve $X$, and $t\mapsto P_t$ is a family of points, then
\begin{equation}\label{eq:cs}\hat{h}_{f_t}(P_t)=\left(\hat{h}_{f_\eta}(P_\eta)+o(1)\right)h_{X, D}(t),\end{equation}
where $\eta\in X$ is the generic point, $D\in \Div(X)$ is any divisor of degree $1$, and $o(1)\to 0$ as $h_{X, D}(t)\to\infty$. The author~\cite{pi:spec} showed that the estimate \eqref{eq:cs} could be improved to
\begin{equation}\label{eq:varpaper}\hat{h}_{f_t}(P_t)=\hat{h}_{f_\eta}(P_\eta)h_{X, D}(t)+O(1),\end{equation}
for a specific divisor $D=D(f, P)\in \Div(X)\otimes\QQ$ of degree 1, but only for polynomial endomorphisms of $\PP^1$ (remarkably, Ghioca and Mavraki~\cite{dragos} have recently extended this to certain non-polynomial rational functions). 
%It is not \emph{a priori} obvious that \eqref{eq:varpaper} is a direct improvement of \eqref{eq:cs}, since the estimate no longer applies to an arbitrary divisor class, but standard facts about heights allow one to conclude from \eqref{eq:varpaper} that \[\hat{h}_{f_t}(P_t)=\hat{h}_{f_\eta}(P_\eta)h_{X, D}(t)+O(\sqrt{h_{X, D}(t)})\] in general.
  The estimate~\eqref{eq:varpaper} is similar to a result of Tate~\cite{tate} in the context of elliptic surfaces, and  is in some sense best possible, since Weil heights are generally viewed as $O(1)$-equivalence classes. This note provides a generalization of the main result of~\cite{pi:spec} to the case of polynomial maps in several variables. Note that for any morphism $f:\PP^N\to\PP^N$ and any hyperplane $H\subseteq\PP^N$, we have $f^*H\sim dH$ for some integer $d\geq 1$. We will say that $f$ is a \emph{regular polynomial endomorphism with respect to} $H$ if this linear equivalence turns out to be an equality. For a study of the complex dynamics of regular polynomial endomorphisms of $\PP^N$, see Bedford and Jonsson~\cite{bj}.

There is a natural map from the space of regular polynomial endomorphisms of $\PP^N$ to the moduli space of endomorphisms of $\PP^{N-1}$ of the same degree, obtained by restricting to the invariant hyperplane, and we will say that a family $f_t$ over $X$ is \emph{fibral} if and only if its image is contained in a single fibre of this projection. 
\begin{theorem}\label{th:main}
Let $K$ be a number field, let $X/K$ be a projective curve,  let $f_t:\PP^N\to\PP^N$ be a fibral family of regular polynomial endomorphisms relative to $H$ over $X$ defined over $K$, and let $P:X\to\PP^N$. There is a divisor $D=D(f, P)\in\Div(X)\otimes \QQ$ of degree 1 such that
\[\hat{h}_{f_t}(P_t)=\hat{h}_{f_\eta}(P_\eta)h_{X, D}(t)+O(1).\]
\end{theorem}

Although Theorem~\ref{th:main} applies to a specific divisor class, from it one can deduce a general result which improves~\eqref{eq:cs} in the relevant cases. Note that the result is easily extended to divisors of arbitrary degree, by the linearity of the Weil height machine.

\begin{corollary}\label{cor:lang}
Let $K$, $X$, $f$, and $P$ be as in Theorem~\ref{th:main}. Then for any divisor $D\in \Div(X)\otimes \QQ$  we have
\[\deg(D)\hat{h}_{f_t}(P_t)=\hat{h}_{f_\eta}(P_\eta)h_{X, D}(t)+O\left(h_{X, D}(t)^{1/2}\right),\]
where the error term may be further improved to $O(1)$ if $X$ is rational.
\end{corollary}

We note that in the case $N=1$, the hypothesis that family be fibral automatically obtains, since $H$ is a point. In the case $N\geq 2$, this assumption is non-trivial, and is used to ensure sufficient regularity of local heights. Specifically, the fact that all polynomials in one variable behave similarly close to the totally invariant point grants a great deal of uniformity in the construction of dynamical Green's functions. In dimension $N\geq 2$, this is no longer true, but when we restrict attention to families for which it is, a similar regularity of Green's functions emerges. We note that the condition may be presented in a more elementary, but coordinate dependent, fashion. If $f\in K(X)(\PP^N)$ is written in coordinates as
\[f_i(x_0, ..., x_N)=\sum_{j_0+\cdots+j_N=d}c_{i, (j_0, ..., j_N)}x_0^{j_0}\cdots x_N^{j_N}\]
for $0\leq i\leq N-1$, with $c_{i, (j_0, ..., j_N)}\in K(X)$ and $f_N=x_N^d$, then $f$ is polynomial with respect to the hyperplane $H$ defined by $x_N=0$. The family is then fibral if (but not only if) $c_{i, (j_0, ..., j_N)}$ for all indices with $j_N=0$.

\begin{remark}
The condition of fibrality of the family in Theorem~\ref{th:main} seems initially to be quite strong, but we note that eliminating it implies something much stronger. In particular, suppose that Theorem~\ref{th:main} holds for arbitrary families of regular polynomial endomorphisms, and let $f/X:\PP^N\to\PP^N$ be any family of endomorphisms (regular or otherwise). It is easy to see that $f$ is the restriction to the invariant hyperplane of a family $\tilde{f}$ of regular polynomial endomorphisms of $\PP^{N+1}$. If $f=[f_0:\cdots :f_N]$, we may simply set $\tilde{f}=[f_0:\cdots f_N:x_{N+1}^{\deg(f)}]$. It is also easy to check that, for any $P\in \PP^N(X)$ we have $\hat{h}_{f_t}(P_t)=\hat{h}_{\tilde{f}_t}(\tilde{P}_t)$, where $\tilde{P}$ is the image of $P$ under the obvious isomorphism $\PP^N\cong \{x_{N+1}=0\}\subseteq \PP^{N+1}$. In other words, if Theorem~\ref{th:main} can be extended to arbitrary families if regular polynomial endomorphisms of $\PP^N$, for any $N$, then it can also be extended to arbitrary families of morphisms.
\end{remark}

It follows immediately from Theorem~\ref{th:main} that if $f$ and $P$ are as in the theorem, and $\hat{h}_{f_\eta}(P_\eta)>0$, then the set
\begin{equation}\label{eq:specialization}\left\{t\in X(\overline{K}):P_t\text{ is preperiodic for }f_t\right\}\end{equation}
is a set of bounded height, and hence contains only finitely many points of any given algebraic degree over $K$, but in many cases something stronger is true. For instance, Baker and DeMarco~\cite{bd} have shown that if \[f_t(x, y, z)=[x^d+tz^d:  y^d+tz^d:z^d],\] and $P_t=(a, b)$ for fixed $a, b\in K$, then the set defined above is actually finite unless $a^d=b^d$. One might speculate, then, that the set defined in \eqref{eq:specialization} is finite in general when $N\geq 2$, unless the dynamics decomposes into lower-dimensional parts which are dependent in some strong sense.

In Section~\ref{sec:greens} we construct dynamical Green's function associated to polynomial morphisms $\PP^N\to\PP^N$ and show that these local heights are fairly uniform for polynomials which are ``the same at infinity.'' In Section~\ref{sec:proof} we study the variation of these heights in families, giving a proof of Theorem~\ref{th:main}.

\section{Green's Functions}\label{sec:greens}

Let $K$ be a field, and let $M_K$ be a set of valuations on $K$. For each $v\in M_K$, we denote by $\CC_v$ the smallest extension of $K$ which is both algebraically closed, and complete with respect to $v$.  By an \emph{$M_K$-constant $\gamma$} we mean a function $\gamma:M_K\to\RR$ such that $\gamma_v=0$ for all but a finite number of $v\in M_K$. A collection of functions $\phi_v:X(\CC_v)\to\RR$ will be called \emph{$M_K$-bounded} if there is an $M_K$-constant $\gamma$ such that $|\phi_v(t)|\leq \gamma_v$ for all $t\in X(\CC_v)$, for all $v\in M_K$; we will often denote $M_K$-boundedness by $\phi_v(t)=O_v(1)$, but observe that the $O_v(1)$ notation here implies both that the constant bound depends on $v$, and that it vanishes for almost all $v$. For convenience we note that if $\gamma_1, \gamma_2$ are $M_K$-constants then so are $\max\{\gamma_1, \gamma_2\}$, $\min\{\gamma_1, \gamma_2\}$,  and $\gamma_1+c\gamma_2$, for any $c\in \RR$.

 Let $f:\PP^N\to\PP^N$ be a morphism satisfying $f^*H=dH$, for the hyperplane $H=\{x_N=0\}$.  For a point in homogeneous coordinates \[P=[x_0:\cdots:x_N]\in\AA^N=\PP^N\setminus H,\] we define
\[\|P\|_v=\frac{\max\{|x_i|_v:0\leq i\leq N-1\}}{|x_N|_v},\]
which is independent of homogeneous scaling.

By a multi-index $I$ of dimension $N$ we mean a tuple $(i_0, ..., i_N)$, and we write $|I|=i_0+\cdots +i_N$. If $\mathbf{x}=(x_0, ..., x_N)$ is a tuple of variables, and $I=(i_0, ..., i_N)$ is a multi-index, then $\mathbf{x}^I$ is an abbreviation for the monomial $x_0^{i_0}\cdots x_N^{i_N}$. 
Given that $f^*H=dH$, we may assume without loss of generality that $f_N(\mathbf{x})=x_N^d$, and write the other components of $f$ in coordinates as
\[f_i(\mathbf{x})=\sum_{|I|=d} a_{i, I}\mathbf{x}^I.\]
Once this representation is fixed, we set
\[\basin_v(f)=\log^+\max\left\{|a_{i, I}|^{1/I_N}:0\leq i\leq N-1, I_N\neq 0\right\}.\]
We note, to aid the reader's intuition, that if $f$ is a monic polynomial in one variable, and $v$ is non-archimedean, then $\basin_v(f)=\log^+\max\{|\alpha|_v:f(\alpha)=0\}$.

 Following the standard convention, we will define a Green's function $G_{f}:\AA^N_{\CC_v}\to\RR$ by
\[G_{f, v}(P)=\lim_{n\to\infty} d^{-n}\log^+\|f^n(P)\|_v.\]
 It is easy to show that both of these limits always exist, and we present a few other properties of these functions.

\begin{lemma}\label{lem:greensprops}
Let $f:\PP^N\to \PP^N$ be a regular polynomial endomorphism defined over $K$.
 There exists an $M_K$-constant $\gamma$ such that if \[\log\|P\|_v>\basin_v(f)+ \gamma_{ v},\] then
\[G_{f, v}(P)=\log\|P\|_v+O_v(1).\]
Furthermore, both $\gamma$ and the implied $M_K$-constant above depend only on the restriction of $f$ to the invariant hyperplane.
\end{lemma}

\begin{proof}
Write the coordinate functions of $f$ as  $f_i=\sum_{|I|=d} a_{i, I}\mathbf{x}^I\in K[\mathbf{x}]$, as above, and let $g_i(\mathbf{x})=\sum_{I_N=0}a_{i, I}\mathbf{x}^I$. In other words, $g_i(\mathbf{x})= f_i(x_0, ..., x_{N-1}, 0)$, the highest-degree homogeneous summand of $f_i(x_0, ..., x_{N-1}, 1)$, and quantities which depend only on the $g_i$ depend only on the restriction of $f$ to $H$.

We will choose our $M_K$ constant $\gamma$ by insisting first that $\gamma_v\geq 0$ for all $v$, and then increasing the value of $\gamma_v$ as necessary. Let $P_0, ..., P_N\in\CC_v$ with $P_N\neq 0$, let $P=[P_0:\cdots :P_N]$, and for a multi-index $I$ write $P^I$ for $\prod P_k^{I_k}$.

If we have $\log\|P\|_v>\basin_v(f)+\gamma_v$, then for any $i$ and any $I$ with $I_N\neq 0$, we have
\[\log|a_{i, I}|_v\leq\log|a_{i, I}|_v+\gamma_v\leq \log|a_{i, I}|_v+I_N\gamma_v<I_N\log\|P\|_v.\]
It follows from this that
\begin{eqnarray*}
\log |a_{i, I}P^I|_v&\leq&\log|a_{i, I}|_v+(d-I_N)\log\max\{|P_0|_v, ..., |P_{N-1}|_v\}+I_N\log|P_N|_v\\
&\leq& d\log|P_N|_v+(d-I_N)\log\|P\|_v+\log |a_{i, I}|_v\\
&<&d\log\|P\|_v+d\log |P_N|_v-\gamma_v.
\end{eqnarray*}
Now let $\var_{1, v}=\log\binom{n+d-1}{d-1}$ if $v$ is archimedean, and  $\var_{1, v}=0$ otherwise.
For each $i$, still subject to $\log\|P\|_v>\basin_v(f)+\gamma_v$, we have
\begin{eqnarray}
\log|f_i(P)-g_i(P)|_v&=&\log\left|\sum_{I_N\neq 0}a_{i, I}P^I\right|_v\nonumber\\
&\leq& \log\max_{I_N\neq 0}|a_{i, I}P^I|_v+\var_{1, v}\nonumber\\
&<&d\log\|P\|_v+d\log|P_N|_v-\gamma_v+\var_{1, v}.\label{eq:effgee}
\end{eqnarray}

By a standard argument from the Nullstellensatz \cite[?]{hind-silv}, there exists an $M_K$-constant $\var_2$ (depending only on the $g_j$) such that
\begin{equation}\label{eq:geenull}\left|d\log\max\{|P_0|_v, ..., |P_{N-1}|_v\}- \log\max\{|g_0(P)|_v, ..., |g_{N-1}(P)|_v\}\right|\leq \var_{2, v}.\end{equation}
It follows that, for some $0\leq i\leq N-1$,
\begin{equation}\label{eq:geeeye}\log|g_i(P)|_v\geq d\log\|P\|_v+d\log|P_N|-\var_{2, v}.\end{equation}
Now let $\var_{3, v}=\log 2$ if $v$ is archimedean, and $\var_{3, v}=0$ otherwise. Combining \eqref{eq:effgee} with \eqref{eq:geeeye}, we see that for the index $i$ in \eqref{eq:geeeye} we have
\[\log|g_i(P)|_v-\log|f_i(P)-g_i(P)|_v-\var_{3 ,v}> \gamma_v-\var_{1, v}-\var_{2, v}-\var_{3, v},\]
and we increase $\gamma$ to ensure that the right-hand-side is never negative. This now implies $|g_i(P)|_v>e^{\var_{3, v}}|f_i(P)-g_i(P)|_v$, which by the triangle inequality gives 
\[|f_i(P)|_v\geq\frac{1}{e^{\var_{3, v}}}|g_i(P)|_v.\]
It follows that
\begin{eqnarray*}
\log\|f(P)\|_v&\geq& \log|g_i(P)|_v-\var_{3, v}-\log|P_N^d|_v\\
&\geq& d\log\|P\|_v-\var_{2, v}-\var_{3, v}.
\end{eqnarray*}
At the same time, applying \eqref{eq:geenull} in the other direction gives 
\begin{eqnarray*}
\log\|f(P)\|_v&=&\max\{|f_j(P)|_v\}-d\log |P_N|_v\\
&\leq&\max\{\log|f_j(P)-g_j(P)|_v, \log|g_j(P)|_v\} +\var_{3, v}-d\log|P_N|_v\\ 
&\leq& d\log\|P\|_v+\var_{3, v}+\max\{\var_{1, v}-\gamma_v, \var_{2, v}\}.
\end{eqnarray*}

Now set
\[\var_{4, v}=\max\{\var_{3, v}+\max\{\var_{1, v}-\gamma_v, \var_{2, v}\}, \var_{2, v}+ \var_{3, v}\}.\]
Note that $\var_4$ is an $M_K$-constant, and since $\var_1$, $\var_2$, and $\var_3$ are independent of $\gamma$, increasing the value of $\gamma$ does not increase the value of $\var_4$.  (Indeed, since $\gamma\geq 0$ by our initial choice, we might just as easily chosen something like $\var_4=\var_3+\max\{\var_1, \var_2, 0\}$.)

We thus have an $M_K$-constant $\var_4$ depending only on the restriction of $f$ to $H$ such that for all $v$,
\begin{equation}\label{eq:mainforeff}-\var_{4, v}\leq\log\|f(P)\|_v-d\log\|P\|_v\leq \var_{4, v}\end{equation}
whenever $\log\|P\|_v>\basin_v(f)+\gamma_v$. If we suppose that $\gamma\geq \var_4/(d-1)$, then this means $\log\|P\|_v>\basin_v(f)+\gamma_v$ also implies \[\log\|f(P)\|_v\geq d\log\|P\|_v-\var_{4, v}>d\basin_v(f)+d\gamma_v-\var_{4, v}>\basin_v+(f)\gamma_v,\]
and so we may apply iteration to the estimate \eqref{eq:mainforeff}.
A standard telescoping sum argument gives
\begin{eqnarray*}
\left|G_{f, v}(P)-\log\|P\|_v\right|&\leq&\sum_{n=0}^\infty\left|d^{-(n+1)}\log\|f^{n+1}(P)\|_v-d^{-n}\log\|f^n(P)\|_v\right|
\\&\leq& \frac{\var_{4, v}}{d}+\frac{\var_{4, v}}{d^2}+\frac{\var_{4, v}}{d^3}+\cdots\\
&=&\frac{1}{d-1}\var_{4, v}.
\end{eqnarray*}
This proves the lemma.
\end{proof}

If $K$ is a global field and $M_K$ is its set of places, then as usual \cite{call-silv} we define the canonical height on $\PP^N$ associated to $f$ by
\[\hat{h}_f(P)=\lim_{n\to\infty}d^{-n}h(f^n(P)).\]
It is easy to show that, for any finite extension $L/K$ and any $P\in\AA^N(L)$, we have
\[\hat{h}_f(P)=\sum_{v\in M_L}\frac{[L_v:K_v]}{[L:K]} G_{f, v}(P).\]
Specifically, if we have, for every place $v$ of $L$, that $\log^+\|P\|_v=0$ or $\log\|P\|_v>\basin_v(f)+ \gamma_v$, then by Lemma~\ref{lem:greensprops} we have that
\[h_{\PP^N, H}(P)=\sum_{v\in M_L}\frac{[L_v:K_v]}{[L:K]}\log\|P\|_v=\sum_{v\in M_L}\frac{[L_v:K_v]}{[L:K]}G_{f, v}(P)+O(1),\]
where the implied constant depends on $f$ but not on $P$.  But the antecedent condition is preserved by the application of $f$, and hence
\[d^{-n}h(f^n(P))=\sum_{v\in M_L}\frac{[L_v:K_v]}{[L:K]}G_{f, v}(P)+O(d^{-n}),\]
where the implied constant is independent of $n$. Taking $n\to\infty$ eliminates the constant. The assumption that  $\log^+\|P\|_v=0$ or $\log\|P\|_v>\basin_v(f)+ \gamma_v$ for each $v$ can now be removed, since it holds for $f^m(P)$ when $m$ is sufficiently large.

\section{Local heights}\label{sec:lemmas}

With local results now in hand, we establish the technical machinery needed for the proof of the main result. Throughout we work over two global fields, namely $K$, a number field, and the function field of a curve $X/K$. We will write $M_K$ for the set of places of $K$, and we will assume that for each $v\in K$ we have chosen a complete, algebraically closed extension $\CC_v$. We will let $F=\overline{K}(X)$ be the function field of $X$ over $\overline{K}\subseteq\CC_v$. Note that, if we take $X$ to be a smooth complete model of itself, the places $M_F$ of $F$ correspond exactly to points $\beta\in X(\overline{K})$. We normalize the absolute value associated to $\beta$ such that for any $\phi\in F$, $-\log|\phi|_\beta$ is the order to which $\phi$ vanishes at $\beta$.

%By a \emph{distance function} $w_\beta$ at the point $\beta\in X(\overline{K})$ we mean an algebraic function of the form $w_\beta=u^{1/n}$, where $u\in F$ is a rational function of degree $n$, vanishing nowhere but at $\beta$. Strictly speaking, $w_\beta$ is not a well-defined function, but $|w_\beta|_v$ is for any $v\in M_K$, and these are the values we will consider. By the Riemann-Roch Theorem, we may always find a function of some degree vanishing only at $\beta$. For instance, if $X=\PP^1$ and $\beta\neq \infty$, we may simply take $w_{\beta}(t)=t-\beta$. 

We recall from \cite{lang} the definition of a local height with respect to a Cartier divisor $D\in \Div(X)$, which is extended to $\Div(X)\otimes \QQ$ by linearity. A \emph{Cartier divisor} is represented by a collection of pairs $(U_i, g_i)$, where the $U_i$ offer a finite cover of $X$ by affine open sets, and the $g_i$ are rational functions on the $U_i$ subject to the condition that $g_i/g_j$ and $g_j/g_i$ are both regular on $U_i\cap U_j$. This allows us to define, unambiguously, $\ord_\beta(D)=-\log|g_i|_\beta$ for any $i$ with $\beta\in U_i$, and the Weil divisor described by this data is $\sum_{\beta\in X} \ord_\beta(D)[\beta]$.

  A \emph{N\'{e}ron divisor} adds a collection of functions $\alpha_{i, v}:U_i(\CC_v)\to \RR$ which are continuous and locally $M_K$-bounded, and the \emph{local height} associated to this data is given by
\[\lambda_{X, D, v}(t)=-\log|g_i(t)|_v+\alpha_{i, v}(t).\]
The compatibility conditions are set out so that this is well-defined. Although the local height function depends on the N\'{e}ron data associated to the Cartier divisor, one can show \cite[Chapter~10]{lang} both that every Cartier divisor admits such a structure, and that changing the N\'{e}ron divisor associated to a given Cartier divisor will perturb the local height only by an $M_K$-bounded amount.

%say $D=\sum m_\beta [\beta]$, then the $v$-adic local height with respect to $D$ will be
%\[\lambda_{X, D, v}(t)=\sum_{\beta\in X}m_\beta \log^+|w_{\beta}(t)^{-1}|_v.\]
%Note that this definition depends on the choice of distance functions, but it is a standard fact that the difference between local heights defined by two different distance functions is $M_K$-bounded. Indeed, one should more properly think of a collection of local heights as an $M_K$-boundedness equivalence class.

The main theorem of this paper is almost an immediate consequence of the following local result, although in the next section we will see a few technical details emerge. Before stating the theorem, we explain the notation somewhat. Since $M_F$ corresponds to the set of $\overline{K}$-points on $X$, we have for each $\beta\in X(\overline{K})$ a value $G_{f_\eta, \beta}(P_eta)$, where we simplify notation by confusing the point and the corresponding valuation. 
\begin{theorem}\label{th:local}
Let $K$, $X$, $f$, $P$ be as above,, let
\begin{equation}\label{eq:Ddef}D=\sum_{\beta\in X}G_{f_\eta, \beta}(P_\eta)[\beta]\in\Div(X)\otimes \QQ,\end{equation}
 let $U\subseteq X$ an affine open set on which the coordinates of $P$ and the coefficients of $f$ are regular. Then  we have
\begin{equation}\label{eq:mainlocal}G_{f_t, v}(P_t)=\lambda_{X, D, v}(t)+O_v(1)\end{equation}
for all $v\in M_K$ and all $t\in U(\CC_v)$ (where the implied constant vanishes for all but finitely many $v$).
\end{theorem}
We restrict to an affine open subscheme of $X$ since one or both of the sides of \eqref{eq:mainlocal} might be undefined at points where some coordinate of $P$ or some coefficient of $f$ is not regular.
 Also, it is worth noting that it is not \emph{a priori} obvious that $G_{f_\eta, \beta}(P_\eta)\in \QQ$ for all $\beta$, a claim which would be false at archimedean places, but it is evident from the proof that this is indeed the case.

Note that, as stated, the theorem applies when $f$ and $P$ are defined over the number field $K$, but the techniques are entirely local. If one has $f$ and $P$ defined over $\CC_v$ then the same estimate holds, although of course only for the place $v$.

The proof of this comes in a sequence of lemmas. Roughly speaking, Lemma~\ref{lem:nearpoles} establishes an inequality of the form~\eqref{eq:mainlocal} in a neighbourhood of each point in the support of the divisor $D=D(f, P)$ defined in \eqref{eq:Ddef}. Since $\lambda_{X, D, v}$ will be bounded away from this support, it then suffices to show that $G_{f_t, v}(P_t)$ is as well (in the sense of $M_K$-boundedness). Lemma~\ref{lem:badpoints} establishes this in neighbourhoods of any points not in the support of $D$, but at which some coordinate of $P$ or some coefficient of $f$ has a pole. Once a neighbourhood of every such point has been removed, it is easy to show that $G_{f_t, v}(P_t)$ is appropriately bounded on the remainder of $X(\CC_v)$, and this is the content of Lemma~\ref{lem:elsewhere}.

\begin{proof}[Proof of Theorem~\ref{th:local}]
We will start by noting that we may replace $P$ by $f(P)$ if convenient, since we have \[G_{f_t, v}(f_t(P_t))=dG_{f_t, v}(P_t)\] and \[D(f, f(P))=dD(f, P),\] whence \[\lambda_{X, D(f, f(P)), v}=d\lambda_{X, D(f, P), v}.\] In other words, proving Theorem~\ref{th:local} for $f(P)$ instead of for $P$ only decreases the implied constant.
 Without loss of generality, then, we may assume that for any $\beta\in X$ we have either $G_{f_\eta, \beta}(P_\eta)=0$ or $\log\|P_\eta\|_\beta>\basin_\beta(f_\eta)$, from which Lemma~\ref{lem:greensprops} gives $G_{f_\eta, \beta}(P_\eta)=\log\|P_\eta\|_\beta$. We will also assume, since the quantities defined are independent of the choice of homogeneous coordinates, that $P_N=1$, and that $D$ is supported entirely on $K$-rational points (this last step amounting just to a finite extension of the ground field).

Note that it suffices to prove the estimate \eqref{eq:mainlocal} on a arbitrary affine open $V\subseteq U$, since $X$ is quasi-compact, so we will assume that $D$ is given locally on $U$ by $g_U=0$, where $g_U\in\Ocal(U)$ is regular. 

For each $\beta\in X(K)$ we will also fix a function $w_\beta$ with no poles other than at $\beta$.  Given this, we will define
\[\mathcal{D}_v(\beta; r, R)=\left\{t\in X(\CC_v):r<\frac{1}{\deg(w_\beta)}\log |w_\beta(t)^{-1}|_v<R\right\},\]
the open annulus at $\beta$ of inner logarithmic radius $r$ and outer logarithmic radius $R$. For brevity we will set $\mathcal{D}_v(\beta; r)=\mathcal{D}_v(\beta; -\infty, r)$, and we note (to forestall possible confusion) that the radii are logarithmic.

\begin{lemma}\label{lem:allthesame}
For any $\phi\in K(X)$ there is an $M_K$-constant $\gamma$ such that
\begin{equation}\label{eq:allthesame}\log^+|\phi(t)|_v=\frac{\log^+|\phi|_\beta}{\deg(w_\beta)}\log^+|w_\beta(t)|_v+O_v(1)\end{equation} for all $t\in \mathcal{D}_v(\beta; \gamma_{ v})$.
\end{lemma}

\begin{proof}
First we note that if $\phi$ has no pole at $\beta$, then this reduces to the standard fact that $\phi$ is bounded in a $v$-adic neighbourhood. We will assume that $\phi$ has a pole, and hence that $\log^+|\phi|_\beta$ is the order of this pole.

Now,
\[\log\left|\frac{\phi^{\deg(w_\beta)}}{w_\beta^{\log^+|\phi|_\beta}}\right|=O_v(1)\]
on $\mathcal{D}_v(\beta; \var_{ v})$ for some $M_K$-constant $\var$, since the function on the left is regular at $\beta$. Unpacking, this becomes 
 \begin{equation}\label{eq:noplus}\log|\phi(t)|_v=\frac{\log^+|\phi|_\beta}{\deg(w_\beta)}\log|w_\beta(t)|_v+O_v(1).\end{equation}
 Since the implied constant vanishes for all but finitely many places, we may choose an $M_K$-constant $\gamma\leq \var$ so that the right-hand-side is non-negative for all $t\in \mathcal{D}_v(\beta; \gamma_{v})$, eliminating the distinction between \eqref{eq:noplus} and \eqref{eq:allthesame}.
\end{proof}

\begin{lemma}\label{lem:localheight}
For any $\beta\in X$, there is an $M_K$-constant $\gamma$ such that
\[\lambda_{X, D, v}(t)=\frac{\ord_\beta(D)}{\deg(w_\beta)}\log^+|w_\beta(t)|_v+O_v(1)\]
on $\mathcal{D}_v(\beta; \gamma_{ v})$.
\end{lemma}

\begin{proof}
Choose an affine open $V$ containing $\beta$ on which $D$ is locally the divisor of $g_V\in K(X)$, and note that $-\log|g_V(t)|_v-\lambda_{X, D, v}(t)$ is locally bounded and continuous by definition. By Lemma~\ref{lem:allthesame}, there is a $\var_{V}$ such that
\[-\log|g_V(t)|_v=\frac{\log^+|g_V^{-1}|_\beta}{\deg(w_\beta)}\log|w_\beta(t)|_v+O_v(1)\]
on $\mathcal{D}_v(\beta; \var_{V, v})$. The estimate follows from this, since $\ord_\beta(D)=\log^+|g_V^{-1}|_\beta$.

We may now take $\gamma$ to be the minimum of the $\var_{V, v}$ as $V$ varies over the affine opens in some finite cover of $U$.
\end{proof}

\begin{lemma}\label{lem:nearpoles}
Suppose that $\beta\in X$, and that $\log\|P_\eta\|_\beta>\basin_\beta(f_\eta)$. Then there is an $M_K$-constant $\gamma$  such that for $t\in \mathcal{D}_v(\beta; \gamma_{v})$,
\[G_{f_t, v}(P_t)=\frac{G_{f_\eta, \beta}(P_\eta)}{\deg(w_\beta)}\log^+|w_\beta(t)|_v+O_v(1).\]
\end{lemma}

\begin{proof}
By hypothesis there is some $i$ such that for every coefficient $a_{j, I}$ of $f$, the function $P_i$ has a pole at $\beta$ of order greater than that of $a_{j, I}^{1/I_N}$, and that $G_{f_\eta, \beta}(P_\eta)=\log|P_i|_\beta$. In other words, the function $P_i^{I_N}/a_{j, I}\in F$ has a pole of positive order at $\beta\in X$. 
By Lemma~\ref{lem:allthesame}, we may choose an $M_K$-constant $\var_{1, j, I}$ such that for some $M_K$-constant $\var_{2, j, I}$ we have
\[\log\left|P_i(t)\right|_v\geq\frac{1}{I_N}\log\left|a_{j, I}(t)\right|_v+\frac{1}{I_N\deg(w_\beta)}\log\left|\frac{P_i^{I_N}}{a_{j, I}}\right|_\beta\log^+|w_\beta(t)|-\var_{2, j, I, v}\]
 for all $t\in \mathcal{D}_v(\beta; \var_{1, j, I, v})$. In particular, if $\var_3$ is the $M_K$-constant $\gamma$ associated to $f$ by Lemma~\ref{lem:greensprops}, then we can set 
 \[\var_{4, j, I}=\min\left\{\var_{1, j, I}, (\var_{2, j, I}+\var_{3})\frac{I_N}{\log|P_i/a_{j, I}|_\beta}\right\}\]
 to ensure that 
 \[\log\left|P_i(t)\right|_v\geq\frac{1}{I_N}\log\left|a_{j, I}(t)\right|_v+\var_{3, v}\]
 for all $t\in \mathcal{D}_v(\beta; \var_{4, j, I, v})$.
 Now if $\var_{5}=\max_{j, I}\{\var_{4, j, I}\}$ we have
\[\log\|P_t\|_v>\basin_v(f_t)+\var_{3, v}\]
for all $t\in \mathcal{D}_v(\beta; \var_{5, v})$. 
Also note that, taking $\gamma\geq \var_5$ large enough, we have $\|P_t\|_v=|P_i(t)|_v$ for all $t\in \mathcal{D}_v(\beta; \gamma_{v})$. 
By Lemma~\ref{lem:greensprops}, for these values of $t$ we see
\begin{eqnarray*}
G_{f_t, v}(P_t)&=&\log\|P_t\|_v+O_v(1)\\
&=&\log|P_i(t)|_v+O_v(1)\\
&=&\frac{\log|P_i|_\beta}{\deg(w_\beta)}\log^+|w_\beta(t)|_v+O_v(1)\\
&=&\frac{G_{f_\eta, \beta}(P_\eta)}{\deg(w_\beta)}\log^+|w_\beta(t)|_v+O_v(1).
\end{eqnarray*}
\end{proof}

The following lemma is a standard application of the triangle inequality.
\begin{lemma}\label{lem:elsewhere}
For each $v\in M_K$, let $Z_v\subseteq U(\CC_v)$, and suppose that the coefficients of $f$ and the coordinates of $P$ are all $M_K$-bounded on the $Z_v$. Then the functions $t\mapsto G_{f_t, v}(P_t)$ are $M_K$-bounded on the $Z_v$. 
\end{lemma}

\begin{proof}
Suppose that the coefficients $a_{i, I}$ of $f$ satisfy $\log|a_{i, I}(t)|_v\leq \var_{1, v}$ for  $t\in Z_v$, and let $\var_{2, v}=\log\binom{n+d-1}{d-1}$ if $v$ is archimedean, and $\var_{2, v}=0$ otherwise. Then for any $t\in Z_v$, the triangle inequality gives
\[\log\|f_t(P_t)\|_v\leq\log\max\{|a_{i, I}(t)P_t^I|_v\}+\var_{2, v}\leq d\log\|P_t\|_v+\var_{1, v}+\var_{2, v}.\]
By induction we have
\begin{equation}\label{eq:beevee}d^{-n}\log\|f_t^n(P_t)\|_v\leq \log\|P_t\|_v+\frac{1}{d-1}\left(\var_{1, v}+\var_{2, v}\right).\end{equation}
The lemma follows by taking $n\to\infty$, and noting that the coordinates of $P$ are (by hypothesis) $M_K$-bounded on the $Z_v$.
\end{proof}

\begin{lemma}\label{lem:badpoints}
Suppose that $\beta\in X(\overline{K})$ and that $G_{f, \beta}(P)=0$. Then there in an $M_K$-constant $\gamma$ such that we have $G_{f_t, v}(P_t)=O_v(1)$ for $t\in \mathcal{D}_v(\beta; \gamma_v)$ .
\end{lemma}

\begin{proof}
First we establish a bound for any given place $v\in M_K$, and later we will show that this bound is almost always zero (and hence the local bounds piece together to an $M_K$-constant bound).

We will choose a uniformizer $\pi\in F$ at $\beta$, and note that any function $\phi\in F$ admits a Laurent series expansion
\[\phi=a_m\pi^m+a_{m+1}\pi^{m+1}+\cdots,\]
where $a_i\in \overline{K}$, $m\in \ZZ$, and $a_m\neq 0$ (from which $m=-\log|\phi|_\beta$). Now, choose $0<\epsilon<1$ so that the series for the coordinates of $P_i$ and the coefficients of $f$ all converge on $0<|\pi|_v<\epsilon^{1/2}$. Since $\epsilon<\epsilon^{1/2}$, these functions all have a maximum $v$-adic absolute value on \[S_\epsilon=\{z\in X(\CC_v):|\pi(z)|_v=\epsilon\}.\]
Let $P_{i, n}$ denote the $i$th coordinate of $f^n(P)$. 
In particular, by \eqref{eq:beevee}  there is a constant $B_v\geq 0$ such that $\log|P_{i, n}(t)|_v\leq d^nB_v$ on $S_\epsilon$.
Also,
\[\pi^{\log^+|P_{i, n}|_\beta}P_{i, n}\in F\]
is regular on $|\pi|_v\leq \epsilon$, and so attains its maximum value on $S_\epsilon$. It follows, if
 \[A_{\delta, \epsilon}=\{z\in X(\CC_v):\delta\leq |\pi(z)|_v\leq\epsilon\},\]
for any given $0<\delta<\epsilon$, that
 \[\delta^{\log^+|P_{i, n}|_\beta}|P_{i, n}(t)|_v\leq \left|\pi(t)^{\log^+|P_{i, n}|_\beta}P_{i, n}(t)\right|_v\leq \epsilon^{\log^+|P_{i, n}|_\beta}d^nB_v,\]
for $t\in A_{\delta, \epsilon}$, whence
 \[\log|P_{i, n}(t)|_v\leq \log^+|P_{i, n}|_\beta\log(\epsilon/\delta)+d^nB_v.\]
 Since this is true for all $i$, we have
 \[G_{f_t, v}(P_t)=\lim_{n\to\infty}d^{-n}\log\|f^n_t(P_t)\|_v\leq  \lim_{n\to\infty}d^{-n}\log\|f_\eta^n(P_\eta)\|_\beta\log(\epsilon/\delta)+B_v=B_v\]
for $t\in A_{\delta, \epsilon}$, since $G_{f_\eta, \beta}(P_\eta)=0$. But this bound is independent of $\delta$, and hence holds on
\[\bigcup_{0<\delta<\epsilon}A_{\delta, \epsilon}=S_{\epsilon}\setminus\{\pi=0\}.\]
 By Lemma~\ref{lem:allthesame} there is a $\gamma_v\in\RR$ such that $\mathcal{D}_v(\beta; \gamma_v)\subseteq S_{\epsilon}\setminus\{\pi=0\}$. In other words, we have $G_{f_t, v}(P_t)$ bounded on some set of the form $\mathcal{D}_v(\beta; \gamma_v)$.

It remains to show that for all but finitely many $v\in M_K$, we may take $B_v=0$ in the above argument, and $\gamma_v=0$. 

Fix a function $\phi$ on $X$.
If $v$ is a non-archimedean place at which the coefficients of the series defining $\phi$ are integral, and the leading coefficient is a unit, then this series converges on $0<|\pi(t)|_v<1$. Note that this is a condition met for all but finitely many places, by Lemma~7 of \cite{pi:spec}. So, given that $v$ is a place of this description, we have \[|\phi(t)|_v=|\pi(t)|_v^{-\log|\phi|_\beta}\] for all $t$ with $0<|\pi(t)|_v<1$. So, applying this to each coefficient of $f$ and each coordinate of $P$, and assuming that the series for each meets the conditions of $v$-integrality, we see that every one of these functions $\phi$ satisfies
$|\phi(t)|_v=\epsilon^{-\log|\phi|_\beta}$ on $S_\epsilon$, for any $0<\epsilon<1$, and so in particular we may take $B_v=-\kappa\log\epsilon$ for some constant $\kappa$ which is independent of $\epsilon$. This now gives the bound \[G_{f_t, v}(P_t)\leq -\kappa\log\epsilon\] for $t\in A_{\delta, \epsilon}$, for any $0<\delta<\epsilon<1$. As above, since this bound is independent of $\delta$, it holds on $A_{0, \epsilon}$. But note that, for $\epsilon\leq \epsilon'<1$, we have $A_{0, \epsilon}\subseteq A_{0, \epsilon'}$, and so we have $G_{f_t, v}(P_t)\leq -\kappa\log\epsilon'$ on $A_{0, \epsilon}$. This holds for an arbitrary $\epsilon\leq\epsilon'<1$, and so taking $\epsilon'\to 1$, we have $G_{f_t, v}(P_t)=0$ on $A_{0, \epsilon}$. This is now independent of $\epsilon$, and so we have the same on $A_{0, 1}$.

So, for all but a finite set of places $v$, we have $G_{f_t, v}(P_t)=0$ on $0<|\pi(t)|_v<1$. But by Lemma~\ref{lem:allthesame}, for all but finitely many $v$ the latter set coincides with $\mathcal{D}_v(\beta; 0)$.
\end{proof}

We are now in a position to finish the proof of Theorem~\ref{th:local}.
Let $S\subseteq X$ be the set of points at which any coordinate of $P$ or any coefficient of $f$ has a pole, and let $U=X\setminus S$. Note that $\operatorname{Supp}(D)\subseteq S$, from Lemma~\ref{lem:greensprops} if nothing else.
For each $\beta\in \Supp(D)$, Lemmas~\ref{lem:localheight} and~\ref{lem:nearpoles} allow us to choose $M_K$-constants $\gamma_\beta$ and $\var_{1, \beta}$ such that
\begin{equation}\label{eq:boundzo}\left|G_{f_t, v}(P_t)-\lambda_{X, D, v}(t)\right|\leq \var_{1, \beta, v}\end{equation}
on $\mathcal{D}_v(\beta; \gamma_{\beta, v})$. We let $\var_1=\max_{\beta\in\operatorname{Supp}(D)}\var_\beta$.

Now, by standard arguments for local heights, we may choose an $M_K$-constant $\var_
2$ such that
\begin{equation}\label{eq:localheightsawayfromsupport}\lambda_{X, D, v}(t)\leq \var_{2, v}\end{equation}
whenever \[t\not\in Z:=\bigcup_{\beta\in\operatorname{Supp}(D)}\mathcal{D}_v(\beta; \gamma_{\beta, v}).\] 
Specifically if $D$ is given locally on the affine patch $V\subseteq U$ by $g_V$ (regular since $D$ is effective), then Lemma~8 of \cite{pi:spec} ensures the existence of an $M_K$-divisor $\var_V$ such that $|g_V(t)|_v<\var_{V, v}$ implies $|w_\beta(t)^{-1}|_v<\deg(w_\beta)\gamma_{\beta, v}$ for some $\beta\in\operatorname{Supp}(D)$. This bounds $\lambda_{X, D, v}(t)$ for $t\in V\setminus Z$, and combining the bounds for some finite affine cover of $X$ gives the claim.
 It suffices to similarly bound $G_{f_t, v}(P_t)$ on $U(\CC_v)\setminus Z$. 

Lemma~\ref{lem:badpoints} provides, for each $\beta\in S\setminus\Supp(D)$,  constants $\gamma_\beta$ and $\var_{3, \beta}$ such that
\begin{equation}\label{eq:greensatbad}G_{f_t, v}(P_t)\leq \var_{3, \beta, v}\end{equation}
for $t\in \mathcal{D}_v(\beta; \gamma_{\beta, v})$. But again by Lemma~8 of \cite{pi:spec}, the coordinates of $P$ and the coefficients of $f$ are $M_K$-bounded on the complement of $\bigcup_{\beta\in S}\mathcal{D}_v(\beta; \gamma_{\beta, v})$, and so by Lemma~\ref{lem:elsewhere} we have an $M_K$-constant $\var_4$ such that
\begin{equation}\label{eq:greenselsewhere}G_{f_t, v}(P_t)\leq \var_4\end{equation}
on this set.

Now, take
\[C=\max\left\{\max_{\beta\in\operatorname{Supp}(D)}\{\var_{1, \beta}\}, \max_{\beta\in S\setminus\Supp(D)}\{\var_2+\var_{3, \beta}\}, \var_2+\var_4\right\},\]
an $M_K$-constant. By \eqref{eq:boundzo}, \eqref{eq:localheightsawayfromsupport}, \eqref{eq:greensatbad}, and \eqref{eq:greenselsewhere} we have
\[\left|G_{f_t, v}(P_t)-\lambda_{X, D, v}(t)\right|\leq \var_{v}\]
for all $t\in U$.
\end{proof}

%\begin{remark}
%The function $t\mapsto G_{f_t, v}(P_t)$ is continuous, in the $v$-adic topology. In particular, if we return to our hypothesis that $\log\|P\|_\beta>\basin_\beta(f)$ for all $\beta$ with $G_{f, \beta}(P)>0$, then there is another way of stating Theorem~\ref{th:local}. Specifically, in that case the data
%\[g_U=P_n\]
%and
%\[\alpha_{U, v}(t)=G_{f_t, v}(P_t)-\log^+|P_n(t)|_v\]
%defines a N\'{e}ron divisor corresponding to the Cartier divisor $\sum G_{f, \beta}(P)[\beta]$.
%\end{remark}

\section{Proof of the main result}\label{sec:proof}

Before proceeding with the bulk of the proof of Theorem~\ref{th:main}, we note that there is no injunction against the image of $P:X\to\PP^N$ landing entirely in the invariant hyperplane, in which case the Green's functions developed in Section~\ref{sec:greens} will be entirely useless. Fortunately, in this special case Theorem~\ref{th:main} reduces to well-known facts about heights.
\begin{lemma}\label{lem:isotrivcase}
Let $f:\PP^N\to\PP^N$ be a morphism defined over $K$, and let $P:X\to\PP^N$ be a morphism. Then
\[\hat{h}_f(P_t)=\hat{h}_{f}(P_\eta)h_{X, D}(t)+O(1),\]
for some divisor $D\in\Div(X)\otimes\QQ$ of degree 1.
\end{lemma}

\begin{proof}
Using the basic properties of the canonical height~\cite{call-silv}, and the functoriality of heights with respect to morphisms~\cite[p.~184]{hind-silv}, we have
\begin{eqnarray*}
\hat{h}_f(P_t)&=&h_{\PP^N, H}(P_t)+O_f(1)\\
&=&h_{X, P^*H}(t)+O_{f, P}(1).
\end{eqnarray*}
It remains to show that $\deg(P^*H)=\hat{h}_{f}(P_\eta)$ in this special case. But this is straightforward in light of the results above. In particular, $\basin_v(f)=0$ for any place $v$ of the function field $K(X)$, and hence $G_{f, v}(P)=\log^+|P_\eta|_v$ for any place $v$. Summing over all places, we see that $\hat{h}_f(P_\eta)=h(P_\eta)=\deg(P)$.
\end{proof}

Now, if the map $P:X\to\PP^N$ lands entirely in the hypersurface $H\subseteq\PP^N$, we may compute the canonical height by restricting the problem to $H$. Our hypothesis that the family is fibral ensures, after a change of variables, that it retricts to a constant family of maps on $H\cong \PP^{N-1}$. Since the canonical height is invariant under change of coordinates, and since composing with this change of coordinates preserves the condition that the image of $P$ lies in $H$, Lemma~\ref{lem:isotrivcase} now gives us Theorem~\ref{th:main}. So we may assume from here on out that the image of $P$ is not contained in $H$, and so there are in fact only finitely many points $t\in X$ such that $P_t\in H$.

Given this restriction, we now define
\[D(f, P)=\sum_{\beta\in X} G_{f_\eta, \beta}(P_\eta)[\beta]\in\Div(X)\otimes \QQ\]
as above.
We note that, by the basic properties of the Green's functions established in Section~\ref{sec:greens}, $D(f, P)$ is an effective divisor of degree $\hat{h}_{f_\eta}(P_\eta)$.  Theorem~\ref{th:main} follows from summing the estimate in Theorem~\ref{th:local} over all places of $K$, or any finite extension thereof, and by taking the divisor in the statement to be $D(f, P)/\deg(D(f, P))$.

Corollary~\ref{cor:lang} follows from the fact \cite[Proposition 5.4, p.~115]{lang} that, for any two divisors $D, E$ of the same degree on $X$, we have
\[h_{X, D}(t)=h_{X, E}(t)+O\left(h_{E, X}(t)^{1/2}\right),\]
with an improvement in the error to $O(1)$ when $X$ is rational.

\end{document}